\documentclass[12pt]{article}
\usepackage{amsmath}
\usepackage{float}
\usepackage{amsthm}
\usepackage{amssymb}
\usepackage[utf8x]{inputenc}
\usepackage{caption}
\usepackage{graphics}
\usepackage{enumerate} 
\usepackage{algorithmic}
\usepackage[Algorithm,ruled]{algorithm}
\usepackage{titlesec}
\usepackage{sectsty}
\titleformat*{\section}{\LARGE\bfseries}
\titleformat*{\subsection}{\Large\bfseries}
\titleformat*{\subsubsection}{\large\bfseries}
\sectionfont{\Huge}
\subsectionfont{\large}
\subsubsectionfont{\normalsize}
\usepackage{multicol}
\usepackage{lmodern}
\usepackage{marvosym} 
\usepackage{lipsum}
\usepackage{mwe}
\usepackage{caption}
\usepackage{subcaption} 
\newtheoremstyle{case}{}{}{}{}{}{:}{ }{}
\theoremstyle{case}

\newcommand{\be}{\begin{equation}}
\newcommand{\ee}{\end{equation}}
\newcommand{\ben}{\begin{eqnarray*}}
\newcommand{\een}{\end{eqnarray*}}
\newtheorem{examp}{\sc Example}
\newtheorem{remk}{\sc Remark}
\newtheorem{corol}{\sc Corollary}
\newtheorem{lemma}{\sc lemma}
\newtheorem{theorem}{\sc theorem}
\newtheorem{defn}{\sc definition}
\newcommand{\bt}{\begin{theorem}}
\newcommand{\et}{\end{theorem}}
\newcommand{\bl}{\begin{lemma}}
\newcommand{\el}{\end{lemma}}
\newcommand{\bed}{\begin{defn}}
\newcommand{\eed}{\end{defn}}
\newcommand{\brem}{\begin{remk}}
\newcommand{\erem}{\end{remk}}
\newcommand{\bex}{\begin{examp}}
\newcommand{\eex}{\end{examp}}
\newcommand{\bcl}{\begin{corol}}
\newcommand{\ecl}{\end{corol}}

\newcommand{\NI}{\noindent}

\newcommand{\vsp}{\vskip 0.5em}

\newcommand{\lam}{t}

\theoremstyle{definition}
\theoremstyle{remark}

\numberwithin{equation}{section}
\numberwithin{theorem}{section}
\numberwithin{lemma}{section}

\begin{document}
\title{More on homotopy continuation method and discounted zero-sum stochastic game with ARAT structure}
\author{ A. Dutta$^{a, 1}$ and A. K. Das$^{b, 2}$\\
\emph{\small $^{a}$Department of Mathematics, Jadavpur University, Kolkata, 700 032, India}\\	
\emph{\small $^{b}$SQC \& OR Unit, Indian Statistical Institute, Kolkata, 700 108, India}\\
\emph{\small $^{1}$Email: aritradutta001@gmail.com}\\
\emph{\small $^{2}$Email: akdas@isical.ac.in} \\
 }
\date{}
\maketitle

\date{}
\maketitle
\begin{abstract}
	 \NI In this paper, we introduce a homotopy function to trace the trajectory by applying modified homotopy continuation method for finding the solution of two-person zero-sum discounted stochastic ARAT game. We show that the algorithm has the higher order of convergence. For the proposed algorithm, the homotopy path approaching the solution is smooth and bounded.  Two numerical examples are illustrated to show the effectiveness of the proposed algorithm. \\
\NI{\bf Keywords:} Two-person zero-sum stochastic game, discounted ARAT stochastic game, homotopy method, optimal value, optimal stationary strategy. \\

\NI{\bf AMS subject classifications:} 91A05, 91A15, 90C33, 90C30,  14F35.
\end{abstract}
\footnotetext[1] {Corresponding author}

\section{Introduction}

In this paper, we consider  two-person  zero-sum discounted stochastic ARAT game. Shapley \cite{shapley1953stochastic} introduced stochastic game and showed that there exist an optimal value and optimal stationary strategies for a  stochastic game with discounted payoff, which depends only on
the current state and not on the history. There are many applications of stochastic games like search problems, military applications, advertising problems, the traveling inspector model, and various 
economic applications. For details see \cite{filar2012competitive}.
There are significant research on theoretical as well as  computational aspects of stochastic games. For details see \cite{sinha1989contribution}, \cite{raghavan1985stochastic}, \cite{raghavan1991stochastic},\cite{sobel1971noncooperative},\cite{solan2015stochastic}.
 Raghaban et al. \cite{raghavan1985stochastic} studied ARAT(Additive Rewards Additive Transition) games and showed that  for a $\beta$-discounted zero-sum ARAT game, the value exists and both players have stationary optimal strategies, which may also be taken as pure strategies. 
 A stochastic game is said to be an Additive Reward Additive Transition
game (ARAT game) if the reward and and the transition probabilities satisfy\\ 

\NI (i) $ r(s, i, j) = r^1 _i (s)+ r^2 _j (s)$ for  $i \in A_s, j \in B_s, s \in S.$\\
\NI(ii) $p_{ij} (s, s') = p^1_i(s,s')+p^2_j(s,s')$ for  $i \in A_s, j \in B_s, (s,s') \in S \times S.$\\

We denote the matrix $((p^1_i(s,s'), s, s' \in S, i \in A_s)))$ as $P_1(s)$ where $S$ is the set of states. This is a $m_1(s) \times d$ matrix where $m_1(s)$ is the cardinality of $A_s$ and $d$ is the cardinality of $s$. Similarly the matrix $P_2(s)$ of order $m_2(s) \times d$ is defined where $m_2(s)$ denotes the cardinality of the set $B_s$.
The Shapley equations for state $s, s' \in S$ can be stated as\\
\begin{center}
\NI$\,\,\,\,\,\,\,\text{Val} [r(s, i, j) + \beta \sum_{s'}p_{ij} (s, s')v_\beta (s')]= v_\beta(s).$
\end{center}
This implies for player I: For any fixed $j$

\begin{equation}\label{inn1}
     r(s, i, j) + \beta \sum_{s'}p_{ij} (s, s')v_\beta(s') \leq v_\beta(s) \ \ \forall \ i.
         \end{equation}
     
\vsp
 
 \NI For playeer II: For any fixed $i$
 
 \begin{equation}\label{inn2}
    r(s, i, j) + \beta \sum_{s'}p_{ij} (s, s')v_\beta(s') \geq v_\beta(s) \ \ \forall \ j 
    \end{equation}
    
  Various approaches have been proposed  for solving different
classes of stochastic games. One such approach is to formulate the ARAT game as complementarity problem. The well-known Lemke's algorithm solves LCPs when the underlying matrix class belongs to a particular class. Cottle and Dantzig\cite{cottle1967complementary} extended Lemke’s algorithm to VLCPs. The processability of  Lemke’s algorithm and Cottle-Dantzig’s algorithm is restricted on some classes of matrices. For details see \cite{eaves1971linear},\cite{garcia1973some}. One sufficient condition for Lemke-processibility and Cottle-Dantzig processibility is that the underlying matrix  should be both $E_0$ and $R_0$ matrix \cite{kostreva1976direct},\cite{murty1988linear},\cite{pang1995complementarity}.

The concept of a class of globally convergent methods\cite{watson1989globally}, known as the homotopy continuation method is used to  prove the existence of solutions to many economic and engineering problems such as systems of nonlinear equations\cite{eaves1972homotopies}, nonlinear optimization problems, fixed point problems, nonlinear programming, game problem and complementarity problems \cite{watson1989modern}. In this paper, we introduce a homotopy function to solve discounted zero sum stochastic ARAT game based on the modified homotopy continuation method  and  establish the higher order global convergence of the homotopy method. 
\vsp 

\vsp
The paper is organized as follows. Section 2 presents some basic notations and results which will be used in the next section. In section 3, we propose a new homotopy function to find the solution of discounted zero sum stochastic ARAT game. We show that the proposed homotopy function possesses a smooth and bounded homotopy path  to find the solution as the homotopy parameter $t$ tends to $0$. To find the solution of homotopy function we modify predictor corrector steps to increase the order of convergency of the algorithm.  We also find the sign of the positive tangent direction of the homotopy path. Finally, in section 4, we illustrate  numerical examples of ARAT stochastic games to present the effectiveness of the introduced homotopy function. 

\section{Preliminaries}
\noindent We begin by introducing some basic notations used in this paper. We consider matrices and vectors with real entries. $R^n$  denotes the $n$ dimensional real space, $R^n_+$ and $R^{n}_{++}$ denote the nonnegative and positive orthant of $R^n.$ We consider vectors and matrices with real entries. Any vector $x\in R^{n}$ is a column vector and  $x^{T}$ denotes the row transpose of $x.$ $e$ denotes the vector of all $1.$ If $A$ is a matrix of order $n,$ $\alpha \subseteq \{1, 2, \cdots, n\}$ and $\bar{\alpha} \subseteq \{1, 2, \cdots, n\} \setminus \alpha$ then $A_{\alpha \bar{\alpha}}$ denotes the submatrix of $A$ consisting of only the rows and columns of $A$ whose indices are in $\alpha$ and $\bar{\alpha}$ respectively. $A_{\alpha \alpha}$ is called a principal submatrix of A and det$(A_{\alpha \alpha})$ is called a principal minor of $A.$ In this paper we consider the followings:
\begin{center}
$\mathcal{R}=\{x\in R^n:x>0,Ax+q>0\}$ \\ 
$\mathcal{\bar{R}}=\{x \in R^n:x\geq 0, Ax+q \geq 0\}$\\
$\mathcal{R}_1=\mathcal{R} \times R_{++}^n \times R_{++}^n$ \\
$\mathcal{\bar{R}}_1=\mathcal{\bar{R}} \times R_{+}^n \times R_{+}^n.$\\ $\partial{\mathcal{R}_1}$ denotes the boundary of $\bar{\mathcal{R}_1}.$ \\
\end{center}
 \subsection{Linear Complementarity Problem and its Generalization}
The linear complementarity problem \cite{neogy2005principal} is defined as follows: Given square matrix $A\in R^{n\times n}$ and a vector $\,q\,\in\,R^{n},\,$ the linear complementarity problem is to find $w \in R^n$ and $x \in R^n$ such that
\begin{equation}\label{1}
w - Ax = q, w \geq 0, \, x \geq 0,
\end{equation}
\begin{equation} \label{2}
x^tw = 0.
\end{equation}
This problem is denoted as LCP$(q, A).$
Several applications of linear complementarity problems are reported in operations research \cite{pang1995complementarity}, multiple objective programming problems \cite{kostreva1993linear}, mathematical economics and engineering. For details see \cite{ferris1997engineering}, \cite{mohan2001more},  \cite{jana2019hidden} and \cite{jana2021more}. The linear complementarity problem is well studied in the literature of mathematical programming and arises in a number of applications in operations research, control theory, mathematical economics, geometry and engineering. For recent works on this problem and applications see \cite{das2017finiteness}, \cite{article12}, \cite{article11} and \cite{article03} and references therein. In complementarity theory several matrix classes are considered due to the study of theoretical properties, applications and its solution methods. For details see \cite{jana2019hidden}, \cite{jana2021more}, \cite{article1}, \cite{mohan2001more}, \cite{neogy2013weak} and \cite{neogy2005almost} and references cited therein. The problem of computing the value vector and optimal stationary strategies for structured stochastic games is formulated as a linear complementary problem for discounted and undiscounded zero-sum games. For details see \cite{mondal2016discounted}, \cite{neogy2008mixture} and \cite{neogy2005linear}.
\vsp
\subsubsection{Vertical Linear Complementarity Problems}
Cottle and Dantzig \cite{cottle1970generalization} extended the linear complementarity problem to a problem in which the matrix $A$ is not a square matrix. The generalization of the linear complementarity
problem introduced by them is given below: Consider a vertical block matrix $A\in R^{m \times k}(m\geq k),$  $A=\left[\begin{array}{c} 
A_1\\
A_2\\
A_3\\ 
\vdots\\
A_k
	\end{array}\right]$ such that  $A_j \in R^{m_j\times k}, 1\leq j \leq k, \sum_{j=1}^k m_j=m.$ This matrix is called vertical block matrix  of type $(m_1,m_2,\cdots m_k)$ and consider $q\in R^m$ where $m=\sum_{j=1}^{k}m_j,$ the generalized linear complementarity problem is to find $w \in R^m$ and $x \in R^k$ such that 

\begin{equation}\label{3}
w-Ax=q, w\geq 0, x \geq 0,
\end{equation}
\begin{equation} \label{4}
x_j\prod_{i}^{m_j}{w^i}_j,   j=1,2,\cdots k.
\end{equation}
This generalization is  known as vertical linear complementarity problem and denoted by VLCP$(q, A)$.  For further details see \cite{cottle1970generalization}.
The vertical block matrix arises naturally in the literature of stochastic games
where the states are represented by the columns and actions in each state are represented by rows in a particular block. For details see \cite{mohan2001pivoting}, \cite{mohan1996generalized}.

An equivalent square matrix $M$ can be constructed from a vertical block matrix $A$ of type $(m_1,..., m_k )$ by copying $A._{j}, m_j$ times for $j=1,2,\cdots,k.$ Therefore $M._{p} = A._{j}$ $ \forall  p \in J_j$ . LCP$(q, M)$ is called as equivalent LCP of VLCP$(q, A)$. For more details see \cite{mohan1996generalized}, \cite{neogy2012generalized}.
 Mohan et al. \cite{mohan1996generalized} proposed techniques to convert
a VLCP to an LCP and also showed that processibility conditions as well.
 Mohan et al.\cite{mohan1999vertical} formulated zero-sum discounted Additive Reward Additive Transition (ARAT) games as a VLCP.
\begin{defn}
 \cite{mohan1999vertical} $A$ is said to be a vertical block $E(d)$-matrix for some $d > 0$ if VLCP$(d, A)$ has a unique solution $w = d, z = 0.$
\end{defn}
 \begin{defn}
  \cite{mohan1999vertical} A is said to be a vertical block $R_0$-matrix if VLCP$(0, A)$ has a unique solution $w = 0, z = 0.$
 \end{defn}
 We denote the class of vertical block $E(d)$ matrices as VBE$(d)$ the class of vertical block $R_0$ matrices by VB$R_0$.
 
 \subsection{Discounted Stochastic Game with the Structure of Additive Reward and Additive Transition }
Consider a state space $S = \{1, 2, \cdots , N\}$. For each $s \in S,$ consider the finite action sets $A_s = \{1, 2, . . . , m_s\}$ for Player I and $B_s = \{1, 2, . . . , n_s\}$ for Player II. For state $s\in S$ a reward law $R(s) = [r(s, i, j)]$ is an $m_s \times n_s$ matrix whose $(i, j)$th entry is  the payoff from Player II to Player I when Player I chooses an action $i \in A_s$ and player II chooses an action $j \in B_s,$ while the game is being played in state $s$ and the payoff from player I to player II is $-r(s,i,j).$ Let $p_{ij} (s, s')$ denotes the probability of a transition from state $s$ to state $s',$ given that Player I and Player II choose actions $i \in A_s, j \in B_s$ respectively. Then transition law is defined by
\begin{center}
    $p = (p_{ij} (s, s') : (s, s') \in S \times S, i \in A_s, j \in B_s).$ 
\end{center}
Let the game be played in stages $t = 0, 1, 2,\cdots.$ At some stage $t$, the players find themselves in a state $s \in S$ and independently choose actions $i \in A_s, j \in B_s.$ Player II pays Player I an amount $r(s, i, j)$ and at stage $(t + 1)$, the new state is $s'$ with probability $p_{ij} (s, s')$. Play continues at this new state. The players guide the game via strategies and in general, strategies can depend on complete histories of the game until the current stage. We are however concerned with the simpler class of stationary strategies which depend only on the current state $s$ and not on stages. So for Player I, a stationary strategy
$k \in K_s = \{k_i(s) | s \in S, i \in A_s, k_i(s) \geq 0, \sum_{i\in A_s}k_i(s) = 1\}$ indicates that the action $i \in A_s$ should be chosen by Player I with probability $k_i(s)$ when the game is in state $s$.\\
Similarly for Player II, a stationary strategy $l \in L_s = \{l_j (s) | s \in S, j \in B_s,$ \ $l_j (s)\geq 0, \sum_{j\in B_s}l_j (s) = 1\}$ indicates that the action $j \in B_s$ should be chosen with probability $l_j (s)$ when the game is in state $s$. Here $K_s$ and $L_s$ will denote the set of all stationary strategies for Player I and Player II respectively. Let $k(s)$ and $l(s)$ be the corresponding $m_s$ and $n_s$ dimensional vectors respectively.
 Fixed stationary strategies $k$ and $l$ induce a Markov chain on $S$ with
transition matrix $P(k, l)$ whose $(s, s')$th entry is given by \\
\begin{center}
 $P_{ss'} (k, l) = \sum_{i\in A_s} \sum_{j\in B_s}p_{ij} (s, s')k_i(s)l_j (s)$
\end{center}
and the expected current reward vector $r(k, l)$ has entries defined by\\
\begin{center}
    $r_s(k, l) = \sum_{i\in A_s}\sum_{j\in B_s}r(s, i, j)k_i(s)l_j (s) = k^t(s)R(s)l(s)$ .
    \end{center}
With fixed general strategies $k, l$ and an initial state $s$, the stream of
expected payoff to Player I at stage $t$, denoted by $v^t _s(k, l), t = 0, 1, 2,\cdots $ is well defined and the resulting discounted  payoff is 
$\phi^\beta _s(k, l)=\sum_{0}^{\infty}\beta^t v^t _s(k, l)$ for a $\beta \in (0,1),$ where $\beta$ is the discount factor.  Due to this additive property assumed on the transition and reward functions, the game is called $\beta$-discounted zero-sum ARAT(Additive Reward Additive Transition) game. For futther details see \cite{goeree1999stochastic}, \cite{raghavan1985stochastic}, \cite{flesch2007stochastic}.
 \vsp
\begin{theorem}
  \cite{filar2012competitive}For ARAT stochastic games\\
(i) Both players possess $\beta$ discounted optimal stationary strategies that are pure.\\
(ii) These strategies are optimal for the average reward criterion as well.\\
(iii) The ordered field property holds for the discounted as well as the average reward
criterion.
 \end{theorem} 
 Now we  observe the following property of the additive components $P_1$ and $P_2$ of the transition probability matrix $P.$ For details see \cite{mohan1999vertical}.
  \begin{lemma}
  If $p^2 _j(s,s') = 0$  for all $s' \in S$ and for some $j \in B(s)$, then $P_2(s) = 0.$
 \end{lemma}
 \begin{theorem}
  \cite{mohan1999vertical}Consider the vertical block matrix $A$ arising from the zero-sum ARAT game. Then $A \in$ VB$E(e)$ where $e$ is the vector each of whose entries is $1.$
 \end{theorem}

\begin{theorem}
\cite{mohan1999vertical} Consider the vertical block matrix $A$ arising from zero-sum ARAT game. Then $A \in $VB$R_0$ if either the condition $(a)$ or the set of conditions $(b)$ stated below is satisfied.\\
$(a)$ \ For each $s$ and each $j \in B_s, \ p^2 _j(s,s) > 0.$\\
$(b) \ \ (i)$ For each $s$, the matrix $P_1(s)$ does not contain any zero column and

$~(ii)$the matrix $P_2(s)$ is not a null matrix.
\end{theorem}

\subsection{Homotopy Continuation Method}
The key idea to solve a system of equations by the homotopy method is to solve  $H(x,t)=0,$ where $H:R^n\times [0,1] \to R^n, x\in R^n, t \in [0,1]$ is called homotopy parameter. The homotopy method aims to trace the entire path of equilibria in $H^{-1}=\{(x,t): H(x,t)=0\}$ by varying both $x$ and $t.$ A parametric path is obtained from a set of functions $(x(s),t(s))\in H^{-1}.$ When we move along the homotopy path, the auxiliary variable $s$ either decreases or increases monotonically. Differentiating $H(x(s),t(s))=0$ with respect to $s$ we obtain $\frac{\partial H}{\partial x}x'(s)+\frac{\partial H}{\partial t}t'(s)=0,$ where $\frac{\partial H}{\partial x}$ and $\frac{\partial H}{\partial t} $ are $n\times n$ Jacobian matrix of $H$ and $n\times 1$ column vector respectively. So this is a system of $n$ differential equations in $n+1$ unknowns ${x_i}'(s) \ \forall \ i$ and $t'(s).$ This system of differential equations has many solutions for which the solutions of the differential equations differ by monotone transformation of the auxiliary variable  $s.$
\vsp
Now we state some results on homotopy which will be required in the next section.
\begin{lemma} \cite{chow1978finding} \label{main}
Let $V \subset R^m$ be an open set and $g :R^m \to R^q$ be smooth. We say $y \in R^q$ is a regular value for $g$ if $\text{Range} \, Dg(x) = R^q $ $\forall x \in g^{-1}(y),$ where $Dg(x)$ denotes the $m \times q$ matrix of partial derivatives of $g(x).$
\end{lemma}
\begin{lemma}\label{par} \cite{Wang}
Let $V \subset R^n, U \subset R^m$ be open sets, and let $\xi:V\times U \to R^l$ be a $C^\alpha$ mapping, where $\alpha >\text{max}\{0,m-l\}.$ If $0\in R^l$ is a regular value of $\xi,$ then for almost all $a \in V, 0$ is a regular value of $\xi _ a=\xi(a,.).$    
\end{lemma}
\begin{lemma}\label{inv} \cite{Wang}
	Let $\xi : V \subset R^m \to R^q$ be $C^\alpha$ mapping, where $\alpha >\text{max}\{0,m-q\}.$ Then $\xi^{-1}(0)$ consists of some $(m-q)$ dimensional $C^\alpha$ manifolds. 
\end{lemma}
\begin{lemma}\label{cl} \cite{N} One-dimensional smooth manifold is diffeomorphic to a unit circle or a unit interval.
\end{lemma}
\begin{lemma}\cite{cordero2012increasing}
 Let $f:R^n\to R^n$ 
be a sufficiently differentiable function in a neighborhood $D$ of $\alpha,$ that is a solution of the system
$f(x) = 0,$ whose Jacobian matrix is continuous and nonsingular in D.  Consider the iterative method $z^k=\phi(x^k,y^k), \ w^k=z^k-f'(y^k)^{-1}f(z^k),$ where $y^k=x^k-f'(x^k)^{-1}f(x^k)$ and $z^k=\phi(x^k,y^k)$ is the iteration function of a method of order $p.$ Then for an initial approximation sufficiently close to $\alpha,$ this
method  has order of convergence $p + 2$.
 \end{lemma}
\begin{lemma}\cite{cordero2012increasing}\label{coorder}
 Consider the function $f:R^n\to R^n$ and the iterative method
 $y^k=x^k-f'(x^k)^{-1}f(x^k),$ \ 
 $z^k=x^k-2(f'(y^k)+f'(x^k))^{-1}f(x^k), \ w^k=z^k-f'(y^k)^{-1}f(z^k)$   has $5$th order convergence.
\end{lemma}
\section{Main results}
In this section, we consider the two-person zero-sum discounted stochstic ARAT game and introduce a  homotopy continuation method to find the solution of the discounted zero-sum ARAT game. We state that a pair of strategies $(k^{\star},l^{\star})$ is optimal for Player I and Player II in the discounted game if for all $s \in S$ $\phi_s(k, l^{\star})\leq \phi_s(k^{\star}, l^{\star}) = v^{\star} _s\leq \phi_s(k^{\star}, l)$
for any strategies $k$ and $l$ of Player I and Player II. The number $v^{\star} _s$
is called the value of the game starting in state $s$ and $v^{\star} = (v^{\star} _1, v^{\star} _2, \cdots, v^{\star} _N )$ is called the value vector. To find the optimal strategy of player I and player II of the two-person zero-sum discounted  ARAT stochastic game  we propose a new function based on the concept of homotopy.\\

Let $X_1,X_2$ be two topological spaces and  $f_1, f_2: X_1 \to X_2$ be continuous maps. A homotopy from $f_1$ to $f_2$ is a continuous function $H:X_1 \times [0,1]  \to X_2,$ defined by $H(x,t)=(1-t)f_1(x)+tf_2(x)$  satisfying $H(x,0) = f_1(x),$ $H(x,1) = f_2(x)  \ \forall x \in X_1.$ The value of $t$ will start from $1$ and goes to $0$ and in this way one can find the solution of the given equation $f_1(x) = 0$ from the solution of $f_2(x) = 0.$\\

 \begin{equation} \label{homf}
H(u,t)=\left[\begin{array}{c} 
(1-t)[(A+A^T)x+q-y_1-A^Ty_2]+t(x-x^{(0)}) \\
Y_1x-t Y_1^{(0)}x^{(0)} +(1-t)X(Ax+q)\\
Y_2(Ax+q)-t Y_2^{(0)}(Ax^{(0)} + q)\\
\end{array}\right]=0
\end{equation} where $Y_1=\text{diag}(y_1),$ $X=\text{diag}(x),$ $Y_2=\text{diag}(y_2),$ $Y_1^{(0)}=\text{diag}(y_1^{(0)}),$ $Y_2^{(0)}=\text{diag}(y_2^{(0)}),$ $u=(x,y_1,y_2) \in R_+^n \times R_+^n \times R_+^n,$ $u^{(0)}=(x^{(0)},{y_1}^{(0)},{y_2}^{(0)})\in \mathcal{R}_1,$ and $\lam \in (0,1].$
\subsection{Computing Solution of ARAT Stochastic Game based on Homotopy Continuation Method}

Now we show that the  solution of the proposed homotopy function will give the solution of discounted ARAT stochastic game. 
\begin{theorem}\label{3.1}
Suppose $\Gamma_u^{(0)}=\{(u,t)\in R^{3n}\times (0,1]: H(u,u^{(0)},t)=0\} \subset \mathcal{R}_1 \times (0,1]\},$  and 
$\cal{A} =\left[\begin{array}{cc} 
-\beta P_1 & E-\beta P_1\\
-E+\beta P_2 & \beta P_2\\
\end{array}\right]$
 and $q =\left[\begin{array}{c} 
-r^1 _i(s)\\
r^2 _j(s)\\
\end{array}\right],$ where the matrix $P_1=P_1(s), P_2=P_2(s)$  and $E=\left[\begin{array}{ccccc} 
e_1 & 0 & 0 & \cdots & 0\\
0 & e_2 & 0 & \cdots & 0\\
\vdots & \vdots & \vdots & \cdots & \vdots\\
0 & 0 & 0 & \cdots & e_d\\
\end{array}\right]$ is a vertical block identity matrix where $e_j, 1\leq j\leq d,$ is a column vector of all $1'$s.
Then the homotopy function 3.1 solves discounted zero-sum stochastic ARAT game. 
\end{theorem}
\begin{proof}
   Suppose for a zero-sum discounted ARAT game the optimal pure strategy in state $s$ is $i_0$ for Player I and $j_0$ for Player II.\\ Then the inequality \ref{inn1} and the inequality \ref{inn2} reduces to\\
   
   \begin{equation}\label{4.1}
r^1 _i (s) + r^2 _{j_0} (s) + \beta \sum_{s'}p^1_i(s,s')v_\beta(s') + \beta\sum_{s'}p^2_{j_0}(s,s')v_\beta(s') \leq v_\beta(s) \ \forall \ i.
\end{equation}
\begin{equation}\label{4.2}
r^1 _{i_0} (s) + r^2 _{j} (s) + \beta \sum_{s'}p^1_{i_0}(s,s')v_\beta(s') + \beta\sum_{s'}p^2_{j}(s,s')v_\beta(s') \geq v_\beta(s) \ \forall \ j.
\end{equation}

\NI The  inequalities \ref{4.1} and \ref{4.2} yield\\
$r^1 _i (s) + \beta \sum_{s'}p^1_i(s,s')v_\beta(s') ≤ v_\beta(s) − \eta_\beta (s) = \xi_\beta (s) \ \forall \ i,$\\ where $\eta_\beta (s) = r^2 _{j_0} (s) + \beta \sum_{s'}p^2_{j_0}(s,s')v_\beta(s')$ and $\xi_\beta (s) = r^1 _{i_0} (s) + \beta \sum_{s'}p^1_{i_0}(s,s')$ and $ \xi_\beta (s)+\eta_\beta (s)=v_\beta(s).$\\
Thus the inequalities are\\
\begin{equation}\label{eq1}
    r^1 _i (s)  + \beta \sum_{s'} p^1_i(s,s')\xi_\beta (s') − \xi_\beta (s) + \beta \sum_{s'} p^1_i(s,s')\eta_\beta (s') \leq 0 \ \forall i \in A_s, s \in S
\end{equation}
and similarly the inequalities for Player II are
\begin{equation}\label{eq2}
   r^2 _j (s)  + \beta \sum_{s'} p^2_j(s,s')\eta_\beta (s') − \eta_\beta (s) + \beta \sum_{s'} p^2_j(s,s')\xi_\beta (s') \geq 0 \ \forall j \in B_s, s \in S .
   \end{equation}

Also for each $s$, in \ref{eq1} there is an $i(s)$ such that equality holds. Similarly, for each $s$ in \ref{eq2} there is a $j(s)$ such that equality holds. Let for $i \in A_s,$
\begin{equation}\label{eq3}
   w^1 _i(s)= −r^1 _i(s) − \beta \sum_{s'} p^1 _i(s, s')\eta\beta(s') + \xi\beta(s) − \beta\sum_{s'} p^1 _i(s, s')\xi\beta(s') \geq 0,
\end{equation}
 and for $j \in B_s$,
 \begin{equation}\label{eq4}
     w^2 _j(s)= r^2 _j(s) − \eta \beta(s) + \beta \sum_{s'} p^2 _j(s, s')\eta \beta(s') + \beta \sum_{s'} p^2 _j(s, s')\xi\beta(s') \geq 0.
 \end{equation} 
We may assume without loss of generality that $\eta\beta(s), \xi\beta(s)$ are strictly positive. Since there is at least one inequality in \ref{eq3} for each $s \in S$ that
holds as an equality and one inequality in \ref{eq4} for each $s \in S$ that holds as an equality, the following complementarity conditions will hold.\\
\begin{equation}\label{eq5}
    \eta_\beta(s)\prod_{i \in A_s}w^1 _i(s) = 0 \ for \ 1 \leq s \leq d 
\end{equation}
and 
\begin{equation}\label{eq6}
    \xi_\beta(s)\prod_{j \in B_s}w^2 _j(s) = 0 \ for \ 1\leq s \leq d .
\end{equation}
The inequality \ref{eq3} and inequality \ref{eq4} along with the complementarity
conditions \ref{eq5} and \ref{eq6} lead to the VLCP$(q, \cal A)$ where the matrix $\cal{A}$ is of the form \\ $\cal{A} =\left[\begin{array}{cc} 
-\beta P_1 & E-\beta P_1\\
-E+\beta P_2 & \beta P_2\\
\end{array}\right]$
 and $q =\left[\begin{array}{c} 
-r^1 _i(s)\\
r^2 _j(s)\\
\end{array}\right],$\\ where  the matrix $P_1=P_1(s), P_2=P_2(s)$  and\\ $E=\left[\begin{array}{ccccc} 
e_1 & 0 & 0 & \cdots & 0\\   
0 & e_2 & 0 & \cdots & 0\\
\vdots & \vdots & \vdots & \cdots & \vdots\\
0 & 0 & 0 & \cdots & e_d\\
\end{array}\right]$ is a vertical block identity matrix where $e_j, 1\leq j\leq d,$ is a column vector of all $1'$s. 
Now an equivalent square matrix $A$ can be constructed from the vertical block matrix $\cal{A}$ of type $(m_1,..., m_c )$ by copying $\cal{A}$$._{j},$  $m_j$ times for $j=1,2,\cdots,c.$ Therefore $A._{p} = \cal{A}$$._{j}$         $ \forall  p \in J_j$ and  the LCP$(q, A)$ is the equivalent LCP of VLCP$(q, \cal A)$.  We consider the proposed homotopy function \ref{homf}\\  
\begin{equation} 
H(u,t)=\left[\begin{array}{c} 
(1-t)[(A+A^T)x+q-y_1-A^Ty_2]+t(x-x^{(0)}) \\
Y_1x-t Y_1^{(0)}x^{(0)} +(1-t)X(Ax+q)\\
Y_2(Ax+q)-t Y_2^{(0)}(Ax^{(0)} + q)\\
\end{array}\right]=0
\end{equation}\\
where $Y_1=\text{diag}(y_1),$ $X=\text{diag}(x),$ $Y_2=\text{diag}(y_2),$ $Y_1^{(0)}=\text{diag}(y_1^{(0)}),$ $Y_2^{(0)}=\text{diag}(y_2^{(0)}),$ $u=(x,y_1,y_2) \in R_+^n \times R_+^n \times R_+^n,$ $u^{(0)}=(x^{(0)},{y_1}^{(0)},{y_2}^{(0)})\in \mathcal{R}_1,$ and $\lam \in (0,1].$ We denote $\Gamma_u^{(0)}=\{(u,t)\in R^{3n}\times (0,1]: H(u,u^{(0)},t)=0\} \subset \mathcal{R}_1 \times (0,1]\}.$ 
  For the proposed homotopy function $t$ varies from $1$ to $0.$ Starting from $t =1$ to $t \to 0$ if we have a smooth  bounded curve, then we obtain a finite solution of the equation \ref{homf} at $t \to 0.$ 
As $t \to 1,$ the  equation \ref{homf} gives the solution $(u^{(0)},1),$ and as $t \to 0,$ the equation \ref{homf} gives the solution of the system of following equations:
\begin{center}
	$(A+A^t)x+q-y_1-A^ty_2=0$\\
	$Y_1x+X(Ax+q)=0$\\
	$Y_2(Ax+q)=0$\\
\end{center}
where $Y_1=\text{diag}(y_1)$ and $Y_2=\text{diag}(y_2).$ Hence the solution of the homotopy function \ref{homf} gives the solution of  discounted zero-sum ARAT game. 
\end{proof}

Therefore if the homotopy function \ref{homf} converges to its solution as the parameter $t \to 0$, we  obtain the solution of discounted ARAT stochastic game.
\vsp

Now we establish the conditions under which the solution exists for the  proposed homotopy function \ref{homf}. We prove the following result to show that the smooth curve $\Gamma_u^{(0)}$ exists for the proposed homotopy function \ref{homf}.

\begin{theorem}\label{reg}
 Let initial point $u^{(0)} \in \mathcal{R}_1.$ Then $0$ is a regular value of the homotopy function $H:R^{3n} \times (0,1] \to R^{3n}$ and the zero point set $H^{-1}(0)=\{(u,t)\in \mathcal{R}_1:H(u,t)=0\}$ contains a smooth curve $\Gamma_u^{(0)}$ starting from $(u^{(0)}, 1).$  
\end{theorem}

\begin{proof}
	The Jacobian matrix of the above homotopy function $H(u, u^{(0)}, t)$ is \\ $DH(u,u^{(0)},t)=$ $\left[\begin{array}{ccc} 
		\frac{\partial{H(u,t)}}{\partial{u}} & 	\frac{\partial{H(u,t)}}{\partial{u^{(0)}}} & \frac{\partial{H(u,t)}}{\partial{t}}\\ 
	\end{array}\right].$ For all $u^{(0)} \in \mathcal{R}_1$ and $t \in (0,1],$ \\ $\frac{\partial{H(u,t)}}{\partial{u^{(0)}}}=$ $\left[\begin{array}{ccc} -\lam I & 0 & 0\\
	-\lam Y_1^{(0)} & -\lam X^{(0)} & 0\\
	-\lam Y_2^{(0)}A &  0 & -t Y^{(0)}\\ 
	\end{array}\right],$ \\where $Y^{(0)}=\text{diag}(Ax^{(0)}+q), X^{(0)}=\text{diag}(x^{(0)}),$ $y^{(0)}=Ax^{(0)}+q$.\\ Now
	$\det(\frac{\partial{H}}{ \partial{u^{(0)}}})=(-1)^{3n}t^{3n}\prod_{i=1}^{n} x_i^{(0)}y_i^{(0)}$ $\neq 0$ for $t \in (0,1].$
	Therefore, $0$ is a regular value of $H(u,u^{(0)},t)$ by the lemma \ref{main}. By lemma \ref{par} and lemma \ref{inv}, for almost all  $u^{(0)} \in \mathcal{R}_1,$ $0$ is a regular value of $H(u,t)$ and $H^{-1}(0)$ consists of some smooth curves and $H(u^{(0)},1)=0.$ Hence there must be a smooth curve  $\Gamma_u^{(0)}$ starting from  $(u^{(0)},1).$
\end{proof}
Hence by Implicit Function Theorem for every $t$ sufficiently close to $1$, the homotopy function has a unique solution $(u,1) $ of \ref{homf}, which is smooth in the parameter  $t$ in a neighbourhood of $(u^{(0)},1) $.\\

 We prove the following result to show that the smooth curve $\Gamma_u^{(0)}$ for the proposed homotopy function \ref{homf} is bounded and convergent. 

\begin{theorem}\label{bnd}
	Let $\mathcal{R}$ be a non-empty set and $A \in R^{n\times n}$  a  matrix and assume that there exists a sequence of points $\{u^k\} \subset \Gamma_u^{(0)} \subset \mathcal{R}_1 \times (0,1],$ where $u^k=(x^k,y_1^k,y_2^k, t^k)$ such that $\|x^k\|< \infty \ \text{as} \ k \to \infty$ and $\|y_2^k\|< \infty \ \text{as} \ k \to \infty$ and for a given $u^{(0)} \in \mathcal{R}_1,$ $0$ is a regular value of $H(u,u^{(0)},t),$ then $\Gamma_u^{(0)}$ is a bounded curve in $\mathcal{R}_1 \times (0,1].$  
\end{theorem}

\begin{proof}
Note that $0$ is a regular value of $H(u,u^{(0)},t)$ by theorem \ref{reg}. By contradiction we assume that $\Gamma_u^{(0)} \subset \mathcal{R}_1 \times (0,1]$ is an unbounded curve. Then there exists a sequence of points $\{v^k\},$ where $v^k=(u^k, t^k) \subset \Gamma_u^{(0)}$ such that $\|(u^k, t^k)\| \to \infty.$ As $(0,1]$ is a bounded set and $x$ component and $y_2$ component of $\Gamma_u^{(0)}$ is bounded, there exists a subsequence of points  $\{v^k\}=\{(u^k, t^k)\}=\{x^k,y_1 ^k,y_2 ^k, t^k\}$ such that $x^k \to \bar{x},\ y_2 ^k \to \bar{y_2}, \ t^k \to \bar{t} \in [0,1] \ \text{and} \ \|y^k\| \to \infty \ \text{as} \ k \to \infty, \ \text{where} \ y^k=\left[\begin{array}{c} y_1 ^k\\ y_2 ^k\\\end{array}\right].$ Since $\Gamma_u^{(0)} \subset H^{-1}(0),$ we have
\begin{equation}\label{zzq}
(1-t^k)[(A+A^T)x^k+q-y_1 ^k-A^T y_2 ^k]+t^k(x^k-x^{(0)})=0 
\end{equation}	
\begin{equation}\label{yyq}
Y_1 ^k x^k-t^k Y_1 ^{(0)}x^{(0)}+(1-t^k)X^k(Ax^k+q)=0
\end{equation}
\begin{equation}\label{wwwq}
Y_2 ^k(Ax^k+q)-t^k Y_2 ^{(0)}(Ax^{(0)}+q)=0
\end{equation}
where $Y_1 ^k=\text{diag}(y_1 ^k), X^k=\text{diag}(x^k)$ and $Y_2 ^k=\text{diag}(y_2 ^k).$\\
\vsp
\NI Let
$\bar{t} \in [0,1], \|y_1 ^k\|=\infty$ and $\|y_2 ^k\|<\infty$ as $k \to \infty.$
Then $\exists \ i \in \{1,2,\cdots, n\}$ such that $y_{1i} ^k \to \infty$ as $k \to \infty.$ Let $I_{1y}=\{i\in\{1,2,\cdots n\} : \lim\limits_{k\to \infty}y_{1i} ^k = \infty\}.$ When $\bar{t} \in [0,1),$ for $i\in I_{1y}$ we write from equation \ref{zzq},\\ $(1-t^k)[((A+A^T)x^k)_i+q_i-y_{1i} ^k-(A^T y_{2} ^k)_i] + t^k(x_i ^k-x_i ^{(0)})=0$\\
$\implies (1-t^k)y_{1i} ^k=(1-t^k)[((A+A^T)x^k)_i+q_i-(A^T y_{2} ^k)_i]+t^k(x_i ^k-x_i ^{(0)}) \\ \implies y_{1i} ^k=[((A+A^T)x^k)_i+q_i-(A^T y_{2} ^k)_i]+\frac{t^k}{(1-t^k)}(x_i ^k-x_i ^{(0)}).$\\ As $k \to \infty$ right hand side is bounded, but left hand side is unbounded. It contradicts that $\|y_1 ^k\|=\infty.$\\ When $\bar{t}=1,$  from equation \ref{yyq}, we obtain, $x_i ^k=\frac{t^k y_{1i} ^{(0)}x_i ^{(0)}}{y_{1i} ^k}$ for $i \in I_{1y}.$ As $k \to \infty, x_i ^k \to 0.$ \\ Again from equation \ref{zzq}, we obtain	$x_i ^{(0)}=\frac{(1-t^k)}{t^k}[((A+A^T)x^k)_i+q_i-y_{1i} ^k-(A^T y_2 ^k)_i]+x_i ^k$ for $i \in I_{1y}.$ As $k \to \infty,$ we have  $x_i ^{(0)}=-\lim\limits_{k\to \infty}\frac{(1-t^k)}{t^k}y_{1i} ^k \leq 0.$ It contradicts that $\|y_1 ^k\|=\infty.$\\ So  $\Gamma_u^{(0)}$ is a bounded curve in $\mathcal{R}_1 \times (0,1].$  
\end{proof}

Therefore the boundedness of the sequences $\{x_k\}$ and $\{y_2 ^k\}$ gurantee the boundedness of the sequence $\{y_1 ^k\},$ i.e. the boundedness of the sequence $\{v^k\}.$  

\begin{theorem}\label{001}
	Suppose the solution set $\Gamma_u^{(0)}$ of the homotopy function $H(u,u^{(0)},t)=0$ is unbounded for $t\in [0,1)$. Then there exists $(\xi, \eta, \zeta) \in R_+^{3n}$ such that $e^t \xi =1,$  $\xi^tA\xi \leq 0.$
\end{theorem}
\begin{proof}
	Let assume that the solution set $\Gamma_u^{(0)}$ is unbounded for $t\in [0,1)$. Then there exists a sequence of points $\{v^k\} \subset \Gamma_u^{(0)} \subset \mathcal{R}_1 \times [0,1),$ where $v^k=(u^k,t^k)=(x^k,y_1^k,y_2^k, t^k)$ such that  $\lim_{k\to \infty}t^k=\bar{t} \in [0,1)$. Now we consider following two cases. \\
	\vsp
	\NI Case 1:  $\|y_2^k\|<\infty$ as $k \to \infty$. Since the solution set $\Gamma_u^{(0)}$ is unbounded we consider the following two subcases.\\ \NI Subcase (i) $\lim_{k\to \infty}e^tx^k=\infty:$\\   Let $\lim_{k\to \infty}\frac{x^k}{e^tx^k}=\xi \geq 0 $ and $\lim_{k\to  \infty}\frac{y_1^k}{e^tx^k}=\eta \geq 0. $ So it is clear that $e^t\xi=1.$ Then dividing by $e^tx^k$ and taking limit $k \to \infty $ from equations \ref{zzq}, \ \ref{yyq}, \ \ref{wwwq} we write\\
	 \begin{eqnarray}
	(1-\bar{t})[(A+A^T)\xi - \eta]+\bar{t}\xi=0\label{n1}\\
	\xi_i \eta_i+\xi_i(A\xi)_i=0 \ \forall \ i \label{n2}
		\end{eqnarray}
		From equations \ref{n1} and \ref{n2} we write $\eta=(A+A^T)\xi+\frac{\bar{t}}{(1-\bar{t})}\xi$ and $ -\xi^TA\xi=\xi^T\eta.$ These two imply that $\xi^T[(A+A^T)\xi+\frac{\bar{t}}{(1-\bar{t})}\xi]=\xi^T\eta=-\xi^TA\xi$ \ for $\bar{t} \in [0,1).$ This implies that $2\xi^TA\xi+\xi^TA^T\xi=-\frac{\bar{t}}{(1-\bar{t})}\xi \leq 0$  i.e. $\xi^TA\xi\leq 0 $ for $\bar{t} \in [0,1).$ \\ Specifically for $\bar{t}=0,$ $\xi^TA\xi= 0$ and for $\bar{t}\in (0,1),$  $\xi^TA\xi<0.$ \\  
	\vsp
		\NI	Subcase (ii) $\lim_{k\to \infty}(1-t^k)e^tx^k= \infty:$\\
	 Let $\lim_{k\to \infty}\frac{(1-t^k)x^k}{(1-t^k)e^tx^k}=\xi'\geq 0.$ Then $e^t\xi'=1.$  Let  $\lim_{k\to \infty}\frac{y_1^k}{(1-t^k)e^tx^k}=\eta' \geq 0.$ Then multiplying the equation \ref{zzq} with $(1-t^k)$ and dividing by $(1-t^k)e^tx^k$, multiplying the equation \ref{yyq} with $(1-t^k)$ and dividing by $((1-t^k)e^tx^k)^2$ and  multiplying the equation \ref{wwwq} with $(1-t^k)$ and dividing by $((1-t^k)e^tx^k)^2$ and taking limit $k \to \infty$, we write
		\begin{eqnarray}
			(1-\bar{t})[(A+A^T)\xi' - (1-\bar{t})\eta']+\bar{t}\xi'=0\label{n11}\\
		\xi'_i \eta'_i+\xi'_i(A\xi')_i=0 \ \forall \ i \label{n22}
		\end{eqnarray}
	Multiplying $(\xi')^T $ in both sides of equation \ref{n11}, we have $(\xi')^t(A+A^T)\xi'- (1-\bar{t})(\xi')^T\eta'=-\frac{\bar{t}}{(1-\bar{t})}(\xi')^T\xi'.$ Now using equation \ref{n22}, we write $(\xi')^T(A+A^T)\xi'+ (1-\bar{t})(\xi')^TA\xi'=-\frac{\bar{t}}{(1-\bar{t})}(\xi')^T\xi' \implies (\xi')^TA^t\xi'+(2-\bar{t})(\xi')^TA\xi'= -\frac{\bar{t}}{(1-\bar{t})}(\xi')^T\xi'$ for $\bar{t} \in [0,1).$ Hence $(3-\bar{t})(\xi')^TA\xi'= -\frac{\bar{t}}{(1-\bar{t})}(\xi')^T\xi'  \implies (\xi')^TA\xi' = -\frac{\bar{t}}{(1-\bar{t})(3-\bar{t})}(\xi')^T\xi'\leq 0.$ So we have $(\xi')^TA\xi'\leq 0$ for $\bar{t} \in [0,1).$ \\Specifically for $\bar{t}=0,$ $(\xi')^TA\xi'= 0$ and for $\bar{t}\in (0,1),$  $(\xi')^TA\xi'<0.$ \\
	
	\vsp
	\NI Case 2:  $\lim_{k\to \infty}e^ty_2^k=\infty$. Since the solution set of $\Gamma_u^{(0)}$ is unbounded we consider following two subcases.\\ \NI Subcase (i) $\lim_{k\to \infty}e^tx^k=\infty$: \\	
		Let $\lim_{k\to \infty}\frac{x^k}{e^Tx^k}=\xi \geq 0, $  $\lim_{k\to  \infty}\frac{y_1^k}{e^Tx^k}=\eta \geq 0 $ and $\lim_{k\to  \infty}\frac{y_2^k}{e^Tx^k}=\zeta \geq 0.$ It is clear that $e^T\xi=1.$ Then dividing by $e^Tx^k$ and taking limit $k \to \infty $ from equation \ref{zzq}, dividing by $(e^Tx^k)^2$ and taking limit $k \to \infty $ from equation \ref{yyq}, \ref{wwwq}, we write
		\begin{eqnarray}
		(1-\bar{t})[(A+A^T)\xi-\eta-A^T\zeta]+\bar{t}\xi=0 \label{nn1}\\
		\xi_i\eta_i+\xi_i(A\xi)_i=0 \ \forall \ i \label{nn2}\\
		\zeta_i(A\xi)_i=0 \ \forall \ i \label{nn3}
		\end{eqnarray} 
		 From equation \ref{nn1}, we have $\eta+A^T\zeta=(A+A^T)\xi+\frac{\bar{t}}{1-\bar{t}}\xi$ for $\bar{t} \in [0,1).$ Now multiplying $\xi^T$ in both sides we get $\xi^T(A+A^T)\xi+\frac{\bar{t}}{1-\bar{t}}\xi^T\xi=\xi^T\eta + \xi^TA^T\zeta.$ From equations \ref{nn2} and \ref{nn3}, we write $\xi^T(A+A^T)\xi+\frac{\bar{t}}{1-\bar{t}}\xi^T\xi= -\xi^TA\xi.$ Hence $\xi^TA\xi+\xi^T(A+A^T)\xi=-\frac{\bar{t}}{1-\bar{t}}\xi^T\xi \leq 0$ for $\bar{t} \in [0,1).$ Therefore $\xi^TA\xi \leq 0$ for $\bar{t} \in [0,1).$ \\ Specifically for $\bar{t}=0,$ $\xi^TA\xi= 0$ and for $\bar{t}\in (0,1),$  $\xi^TA\xi<0.$ \\ 
		 \vsp
	
\NI	Subcase(ii) $\lim_{k\to \infty}(1-t^k)e^tx^k= \infty:$ \\ Let $\lim_{k\to \infty}\frac{(1-t^k)x^k}{(1-t^k)e^Tx^k}=\xi'\geq 0.$ Then $e^T\xi'=1.$  Let  $\lim_{k\to \infty}\frac{y_1^k}{(1-t^k)e^Tx^k}=\eta' \geq 0$ and $\lim_{k\to \infty}\frac{y_2^k}{(1-t^k)e^Tx^k}=\zeta' \geq 0$  Then multiplying the equation \ref{zzq} with $(1-t^k)$ and dividing by $(1-t^k)e^Tx^k$, multiplying the equation \ref{yyq} with $(1-t^k)$ and dividing by $((1-t^k)e^Tx^k)^2$ and  multiplying the equation \ref{wwwq} with $(1-t^k)$ and dividing by $((1-t^k)e^Tx^k)^2$ and taking limit $k \to \infty,$ we obtain
		 \begin{eqnarray}
		 (1-\bar{t})(A+A^T)\xi'- (1-\bar{t})^2\eta'- (1-\bar{t})^2A^T\zeta'+\bar{t}\xi'=0 \label{nnn1}\\
		 \xi'_i\eta'_i+\xi'_i(A\xi')_i=0 \ \forall \ i \label{nnn2}\\
		 \zeta'_i(A\xi')_i=0 \ \forall \ i \label{nnn3}
		 \end{eqnarray}
		 Multiplying $(\xi')^T$ in both side of equation \ref{nnn1}, we have $(\xi')^T(A+A^T)\xi'- (1-\bar{t})(\xi')^T\eta'- (1-\bar{t})(\xi')^TA^T\zeta'=-\frac{\bar{t}}{(1-\bar{t})}(\xi')^T\xi'$. Now from equations \ref{nnn2} and \ref{nnn3}, we write $(\xi')^T(A+A^T)\xi'+ (1-\bar{t})(\xi')^TA\xi'=(\xi')^TA^T\xi'+(2-\bar{t})(\xi')^T\xi'=(3-\bar{t})(\xi')^TA\xi' =-\frac{\bar{t}}{(1-\bar{t})}(\xi')^T\xi'\implies (\xi')^TA\xi' = -\frac{\bar{t}}{(1-\bar{t})(3-\bar{t})}(\xi')^T\xi'\leq 0$ for $\bar{t} \in [0,1).$\\ Specifically for $\bar{t}=0,$ $(\xi')^TA\xi'= 0$ and for $\bar{t}\in (0,1),$  $(\xi')^TA\xi'<0.$\\
		 
			 Hence considering all the cases it is proved that the unboundedness of the solution set $\Gamma_u^{(0)}$ of the homotopy function $H(u,u^{(0)},t)=0$ leads to the existence of  $(\xi, \eta, \zeta) \in R_+^{3n}$ such that $e^t \xi =1,$  $\xi^tA\xi \leq 0$ for $t\in[0,1).$

		\end{proof}

	\begin{corol}\label{00}
		For $\bar{t}=1$, the homotopy curve is bounded. 
		\end{corol}
	\begin{proof}
	    Consider that the homotopy curve is unbounded in the neighbourhood of $\bar{t}=1$. Then there exists a sequence of points $\{v^k\} \subset \Gamma_u^{(0)} \subset \mathcal{R}_1 \times [0,1),$ where $v^k=(u^k,t^k)=(x^k,y_1^k,y_2^k, t^k)$ such that  $\lim_{k\to \infty}t^k=\bar{t}=1$. Now we consider following two cases.\\
	    Case 1: $\lim_{k\to \infty}e^tx^k=\infty.$  \\	Case 2: $\lim_{k\to \infty}(1-t^k)e^tx^k= \infty$.\\
	    
	     \NI  Case 1. Let $\lim_{k\to \infty}e^tx^k=\infty$ and $\lim_{k\to \infty}\frac{x^k}{e^Tx^k}=\xi \geq 0. $ Hence $e^t\xi=1.$ If $\|y_2^k\|<\infty$ as $k \to \infty,$ then from equation \ref{n1}, we get $\xi=0$ for $\bar{t}=1.$ If $\lim_{k\to \infty}e^ty_2^k=\infty,$ then from equation \ref{nn1}, we get $\xi=0$ for $\bar{t}=1.$ This contradicts that $e^t\xi=1.$ 
	       
\NI Case 2. Let $\lim_{k\to \infty}(1-t^k)e^tx^k= \infty$ and $\lim_{k\to \infty}\frac{(1-t^k)x^k}{(1-t^k)e^Tx^k}=\xi'\geq 0.$ Hence $e^t\xi'=1.$ If $\|y_2^k\|<\infty$ as $k \to \infty,$ then from equation \ref{n11}, we get $\xi'=0$ for $\bar{t}=1.$ If $\lim_{k\to \infty}e^ty_2^k=\infty,$ then from equation \ref{nnn1}, we get $\xi'=0$ for $\bar{t}=1.$ This contradicts that $e^t\xi'=1.$ 
	  
     \NI Therefore the homotopy curve is  bounded for $\bar{t}=1.$ 
	\end{proof}
 
\begin{theorem}
	Let $A\in R^{n\times n}$ be a matrix. If the set  $\mathcal{R}_1$ be nonempty and $0$ is a regular value of $H(u,u^{(0)},t),$ then the homotopy path $\Gamma_u^{(0)} \subset \mathcal{R}_1 \times (0,1]$ is bounded.
\end{theorem}
\begin{proof}
Suppose $A\in R^{n\times n}$ is a matrix and  there exists a sequence of points $\{v^k\} \subset \Gamma_u^{(0)} \subset \mathcal{R}_1 \times (0,1],$ where $v^k=(x^k,y_1^k,y_2^k, t^k).$ Hence by the definition of $\mathcal{R}_1$  $x^k,y_1^k,y_2^k,\\Ax^k+q>0.$ From corollary \ref{00} the homotopy curve is bounded for $\lam=1.$ Assume that the homotopy curve $\Gamma_u^{(0)} \subset \mathcal{R}_1 \times (0,1)$ is unbounded. Then from theorem \ref{001}   $(\xi)^TA\xi<0$ for $t\in (0,1).$ But $Ax^k+q>0$ implies that $A\xi \geq 0,$ where $\xi=$ $\lim_{k\to \infty}\frac{x^k}{e^Tx^k} \geq 0$ for $\lim_{k\to \infty}{e^Tx^k}=\infty,$ or $\xi=$ $\lim_{k\to \infty}\frac{(1-t^k)x^k}{(1-t^k)e^Tx^k}\geq 0$ for $\lim_{k\to \infty}{(1-t^k)e^Tx^k}=\infty.$ Hence $\xi, A\xi \geq 0$ imply that $\xi^TA\xi\geq 0$ for $t\in(0,1)$, which contradicts that the homotopy path is unbounded for $t\in(0,1).$ Hence the homotopy curve $\Gamma_u^{(0)} \subset \mathcal{R}_1 \times (0,1]$ is bounded.
\end{proof}
Therefore the homotopy curve $\Gamma_u^{(0)}$ is bounded for the parameter t starting from $1$ to $0$ if the set  $\mathcal{R}_1$ be nonempty and $0$ is a regular value of the homotopy function \ref{homf} .\\

For an initial point $u^{(0)}\in \mathcal{R}_1$ we obtain a smooth bounded homotopy path which  leads to  the solution of homotopy function \ref{homf} as the parameter $t \to 0.$ 
\begin{theorem}\label{3.7}
	For $u^{(0)}=(x^{(0)},y_1^{(0)},y_2^{(0)})\in \mathcal{R}_1,$ the homotopy equation finds a bounded smooth curve $\Gamma_u^{(0)} \subset \mathcal{R}_1 \times (0,1]$ which starts from $(u^{(0)},1)$ and approaches the hyperplane at $t =0.$ As $t\to 0,$ the limit set $L \times \{0\} \subset \bar{\mathcal{R}}_1 \times \{0\}$ of $\Gamma_u^{(0)}$ is nonempty and every point in $L$ is a solution of the following system:
	\begin{equation}\label{sys}
	\begin{split}
	(A+A^T)x+q-y_1-A^Ty_2=0 \\
	Y_1x+X(Ax+q)=0  \\
		Y_2(Ax+q)=0. \\
	\end{split}
    \end{equation}
\end{theorem}
\begin{proof}
  	Note that $\Gamma_u^{(0)}$ is diffeomorphic to a unit circle or a unit interval $(0,1]$ in view of lemma \ref{cl}. As $\frac{\partial{H(u,u^{(0)},1)}}{\partial{u^{(0)}}}$ is nonsingular, $\Gamma_u^{(0)}$ is diffeomorphic to a unit interval $(0,1].$ Again $\Gamma_u^{(0)}$ is a bounded smooth curve by the theorem \ref{bnd}. Let $(\bar{u},\bar{t})$ be a limit point of $\Gamma_u^{(0)}.$ Now consider four cases: $(i)(\bar{u},\bar{t})\in \mathcal{R}_1 \times \{1\}:$ As the equation $H(u,1)=0$ has only one solution $u^{(0)}\in  \mathcal{R}_1, $ this case is impossible.\\
  	$(ii)(\bar{u},\bar{t})\in \partial{\mathcal{R}_1} \times \{1\}:$ there exists a subsequence of $(u^k, t^k) \in \Gamma_u^{(0)}$ such that $x_i^k \to 0$ or $(Ax^k+q)_i \to 0$ for $i \subseteq \{1,2,\cdots n\}.$ From the last two equalities of the homotopy function \ref{homf}, we have $y_1^k \to \infty$ or $y_2^k \to \infty.$ Hence it contradicts the boundedness of the homotopy path by the theorem \ref{bnd}.\\
  	$(iii)(\bar{u},\bar{t})\in \partial{\mathcal{R}_1} \times (0,1):$ Also impossible followed by the case $(ii).$\\
  	$(iv)(\bar{u},\bar{t})\in \bar{\mathcal{R}}_1 \times \{0\}:$ The only possible case.\\

 \NI Hence $\bar{u}=(\bar{x},\bar{y_1},\bar{y_2})$ is a solution of the system \ref{sys}
 \begin{center}
     $(A+A^T)x+q-y_1-A^Ty_2=0$ \\ 
    $ Y_1x+X(Ax+q)=0  $\\ 
    $ Y_2(Ax+q)=0.$
 \end{center}
\end{proof}
Note that theorem  \ref{3.7} establishes the solution of the proposed  homotopy function which  validates the theorem \ref{3.1}. This in turn leads to the solution of discounted ARAT stochastic game.\\
In this approach the initial point $u^{(0)}=(x^{(0)},y_1^{(0)},y_2^{(0)})\in \mathcal{R}_1$  has to be a  feasible point. Hence choose the initial point such that  $x^{(0)}>0, \ Ax^{(0)}+q>0.$ Here $(\bar{u},0)$ is the solution of the homotopy function \ref{homf}. Therefore $\bar{u}\in \bar{R}_1$ is the solution of the system of equations \ref{sys}. Hence $\bar{Y}_1\bar{x}= 0$ and $\bar{X}(A\bar{x}+q)=0,$ where $\bar{Y}_1=$diag$(\bar{y}_1)$ and $\bar{X}=$diag$(\bar{x})$. It is clear that the component $\bar{x}$ of  $\bar{u}=(\bar{x},\bar{y_1},\bar{y_2})$ provides the solution  of discounted ARAT stochastic game.

\subsection{Tracing Homotopy Path}
  We trace the homotopy path $\Gamma_u^{(0)} \subset \mathcal{R}_1 \times (0,1]$ from the initial point $(u^{(0)},1)$ as $t \to 0.$ To find the solution of the discounted ARAT stochastic game we consider homotopy path along with other assumptions. Let $s$ denote the arc length of $\Gamma_u^{(0)}.$ We parameterize the homotopy path $\Gamma_u^{(0)}$ with respect to $s$ in the following form\\
 	\begin{equation}\label{ss}
 	H(u(s),t (s))=0, \ 
 	u(0)=u^{(0)}, \   t(0)=1.
 	\end{equation} 
 	The solution of the equation \ref{ss} satisfies the initial value problem 
 		\begin{equation}\label{sssss}
 	\Dot{u}=-\frac{\partial}{\partial u}H(u,t)^{-1}\frac{\partial}{\partial t}H(u,t), \ 	u(0)=u^{(0)}
 		\end{equation} 
   From equation \ref{homf} the choice of $H$ is $H(u,t)=(1-t)f(u)+tg(u)=0,$ where 
   
   $f(u)=\left[\begin{array}{c} 
		(A+A^T)x+q-y_1-A^Ty_2 \\
         Y_1x +X(Ax+q)\\
          Y_2(Ax+q)\\
	\end{array}\right]$ and $g(u)=\left[\begin{array}{c} 
		x-x^{(0)} \\
Y_1x- Y_1^{(0)}x^{(0)} \\
Y_2(Ax+q)- Y_2^{(0)}(Ax^{(0)} + q)\\ 
	\end{array}\right].$\\ Hence the system \ref{sssss} becomes \\
	\begin{center}
	    $	\Dot{u}=-((1-t)J_f+tJ_g)^{-1}(g(u)-f(u)), \ 	u(0)=u^{(0)}$ 
	    	\end{center} where $J_f$ and $J_g$ are Jacobian matrices of the functions $f$ and $g$.Hence $\Dot{u}=-\tilde{J}^{-1}\tilde{f}, $ where $ \tilde{J}= (1-t)J_f+tJ_g$ and $\tilde{f}=g(u)-f(u)$. \\ The initial value problem \ref{sssss} reduces to\\ 
	    	\begin{center}
	    	    $\Dot{u}=p(u,t), \ 	u(0)=u^{(0)}$ where $p(u,t)=-\tilde{J}^{-1}\tilde{f}$\end{center}
	    	    This problem will be solved by iterative process \\ \begin{center}
	    	        
	    	    $u_{(i+1)}=P(u_i,t_i,h_i),$ where $h_i=t_{i+1}-t_i.$ \end{center}
	    	    Here $u_i$ is an approximation of $u(s).$
	    	    $P(u_i,t_i,h_i)$ is given by  \\ 
	    	    \begin{center}
	    	    $P(u,t,h)=I_m(u,t,h),$ where $I_0(u,t,h)=u$  \\and $K_j=\frac{\partial}{\partial u}H(I_j,t+h)^+H(I_j,t+h)$\\ $L_j=I_j-K_j$\\  $KK_j=(\frac{\partial}{\partial u}H(L_j,t+h)+\frac{\partial}{\partial u}H(I_j,t+h))^+H(I_j,t+h)$\\ $LL_j=I_j-2*KK_j$
	    	     \end{center}
	    	     The next iteration 
	    	     \begin{center}
	    	  $I_{j+1}=LL_j-\frac{\partial}{\partial u}H(L_j,t+h)^+H(L_j,t+h),$ for $j=0,1,2,\cdots,m-1.$ 
	    	  \end{center}
	    	  Therefore in each step finding the value of $I_{j+1}$ for $m$ times we obtain the next iteration and iterative process will continue until the termination criteria is satisfied.
	\newpage
     
\begin{algorithm}
\caption{: \ Modified Homotopy Continuation Method}\label{allgo}

\NI \textbf{Step 0:} Initialize $(u^{(0)},t_0)$ and a natural number $m\in(0,50).$ Set $l_0 \in (0, 1).$ Choose $\epsilon_2 >> \epsilon_3 >> \epsilon_1 > 0$ which are very small positive quantity.
\vsp 
\NI \textbf{Step 1:} $\tau^{(0)}= \xi^{(0)}=(\frac{1}{n_0})\left[\begin{array}{c} 
s^{(0)}\\
-1\\
\end{array}\right]$ for $i=0,$ where $n_0=\|\left[\begin{array}{c} 
s^{(0)}\\
-1\\
\end{array}\right]\|$ and $s^{(0)}= (\frac{\partial H}{\partial u}(u^{(0)},t_0))^{-1}(\frac{\partial H}{\partial t}(u^{(0)},t_0)).$ \\ For $i>0,$ $s^{(i)}= (\frac{\partial H}{\partial u}(u^{(i)},t_i))^{-1}(\frac{\partial H}{\partial t}(u^{(i)},t_i)),$ $n_i=\|\left[\begin{array}{c} 
s^{(i)}\\
-1\\
\end{array}\right]\|,$ $\xi^{(i)}=(\frac{1}{n_i})\left[\begin{array}{c} 
s^{(i)}\\
-1\\
\end{array}\right]$.\\ 
If $\det  (\frac{\partial H}{\partial u}(u^{(i)},t_i))>0,$ $\tau^{(i)}= \xi^{(i)}$ else $\tau^{(i)}= -\xi^{(i)},$ $i \geq 1.$\\ Set $l=0.$
\vsp
\NI \textbf{Step 2:} (Predictor and corrector point calculation) $(\tilde{u}^{(i)},\tilde{t}_i)=(u^{(i)},t_i)+a\tau^{(i)},$ where $a={l_0}^l.$ Compute $(\hat{u}^{(i)},\hat{t}_{i})=H'_{u^{(0)}}(\tilde{u}^{(i)},\tilde{t}_i)^+ H(\tilde{u}^{(i)}, \tilde{t}_i)$ and $(\bar{u}^{(i)},\bar{t}_{i})=(\tilde{u}^{(i)},\tilde{t}_i)-(\hat{u}^{(i)},\hat{t}_{i}).$ Now compute $(\hat{uu}^{(i)},\hat{tt}_{i})=(H'_{u^{(0)}}(\tilde{u}^{(i)},\tilde{t}_i)+H'_{u^{(0)}}(\bar{u}^{(i)},\bar{t}_i))^+ H(\tilde{u}^{(i)}, \tilde{t}_i)$ and $(\bar{uu}^{(i)},\bar{tt}_{i})=(\tilde{u}^{(i)},\tilde{t}_i)-2(\hat{uu}^{(i)},\hat{tt}_{i}).$ \\Compute $(u^{(i+1)},t_{i+1})=(\bar{uu}^{(i)},\bar{tt}_{i})-H'_{u^{(0)}}(\bar{u}^{(i)},\bar{t}_i)^+H(\bar{uu}^{(i)},\bar{tt}_{i}).$ \\Repeat the method from the computation of $(\tilde{u}^{(i)},\tilde{t}_{i})$ to the computation of $(u^{(i+1)},t_{i+1})$ for $m$ times. In each step after repeating the computation for $m$ times, can obtain the value of next iteration. $(u^{(i+1)},t_{i+1})$  \\ If $0<\|{t_{i+1}-t_i}\|<1, $ go to step 3. Otherwise if $m' = \min(a,\|({u}^{(i+1)},{t}_{i+1})-({u}^{(i)},{t}_{i})\|)>a_0,$ update $l$ by $l+1,$ and recompute $(\tilde{t}_i, \hat{t}_{i})$ else go to step 3.
\vsp
\NI \textbf{Step 3:} Determine the norm $r=\|H(u^{(i+1)},t_{i+1})\|.$ If $r \leq 1$ and $u^{(i+1)}>0$ go to step 5, otherwise if $a > \epsilon_3,$ update $l$ by $l+1$ and go to step 2 else go to step 4.
\vsp
\NI \textbf{Step 4:} If $|t_{i+1} - t_i| < \epsilon_2,$ then if $|t_{i+1}| < \epsilon_2,$ then stop with the solution $(u^{(i+1)},t_{i+1}),$ else terminate (unable to find solution) else $i=i+1$ and go to step 1.
\vsp
\NI \textbf{Step 5:} If $|t_{i+1}| \leq \epsilon_1,$ then stop with solution $(u^{(i+1)},t_{i+1}),$ else $i=i+1$ and go to step 1.
\end{algorithm}

\NI Note that in step 2, $H'_{u^{(0)}}(u,t)^+ = H'_{u^{(0)}}(u,{t})^{T}(H'_{u^{(0)}}(u,{t})H'_{u^{(0)}}(u,t)^{T})^{-1}$ is the Moore-Penrose inverse of $H'_{u^{(0)}}(u,t).$ 
\vsp

 We prove the following theorem to obtain the positive direction of the proposed algorithm.

\begin{theorem} \label{direction}
If the homotopy curve $\Gamma_u^{(0)}$ is smooth, then the positive predictor direction $\tau^{(0)}$ at the initial point $u^{(0)}$ satisfies $\det \left[\begin{array}{c}
\frac{\partial H}{\partial u \partial t}(u^{(0)},1)\\
\tau ^{(0)^T}\\
\end{array}\right] < 0.$ 
\end{theorem}
\begin{proof}
	From the equation \ref{homf}, we consider the following homotopy function \\ 
	\vsp
	$H(u,t)=$ $\left[\begin{array}{c} 
	(1-t)[(A+A^T)x+q-y_1-A^Tz_2]+t(x-x^{(0)}) \\
	Y_1x-t Y_1^{(0)}x^{(0)}+(1-t)X(Ax+q)\\
	Y_2(Ax+q)-t Y_2^{(0)}(Ax^{(0)}+q)\\
	\end{array}\right]=0.$\\   
	\vsp
	\NI Now $\frac{\partial H}{\partial u \partial t}(u,t)=$
	
	$\left[\begin{array}{cccc} 
	(1-t)(A + A^T) + t I & -(1-t)I & -(1-t)A^T & Q\\
	Y_1+(1-t)(Y+XA) & X & 0 & -Y_1^{(0)}x^{(0)}-X(Ax+q)\\
	Y_2A & 0 & Y & -Y_2^{(0)}(Ax^{(0)}+q)\\ \
	\end{array}\right],$\\ where $Q = (x-x^{(0)})-[(A+A^T)x+q-y_1-A^ty_2]$ and $Y=\text{diag}(Ax+q).$ \\ 
	\vsp
	\NI At the initial point $(u^{(0)},1)$\\
	 $\frac{\partial H}{\partial u \partial t}(u^{(0)},1)=\left[\begin{array}{cccc} 
	 I & 0 & 0 & -[(A+A^T)x^{(0)}+q-y^{(0)}_1-A^Ty^{(0)}_2]\\
	Y^{(0)}_1 & X^{(0)} & 0 & -Y^{(0)}_1x^{(0)}-X^{(0)}(Ax^{(0)}+q)\\
	Y^{(0)}_2A & 0 & Y^{(0)} & -Y^{(0)}_2(Ax^{(0)}+q)\\ 
	\end{array}\right].$\\
	\vsp
	\NI Let positive predictor direction be $\tau^{(0)}=\left[\begin{array}{c}
	\kappa \\ -1
	\end{array}\right] = \left[\begin{array}{c}
	(Q^{(0)}_1)^{(-1)}Q_2^{(0)} \\ -1
	\end{array} \right],$ \\ where 
	\vsp
	$Q^{(0)}_1=\left[\begin{array}{ccc} 
	I & 0 & 0 \\
	Y^{(0)}_1 & X^{(0)} & 0 \\
	Y^{(0)}_2A & 0 & Y^{(0)} \\ 
	\end{array}\right],$ \\ $Q^{(0)}_2=\left[\begin{array}{c} 
	 -[(A+A^T)x^{(0)}+q-y^{(0)}_1-A^Ty^{(0)}_2]\\
	-Y^{(0)}_1x^{(0)}-X^{(0)}(Ax^{(0)}+q) \\
	-Y^{(0)}_2(Ax^{(0)}+q) \\ 
	\end{array}\right]$ and $\kappa$ is an $n \times 1$ column vector.\\
	 Hence, 
    $\det\left[\begin{array}{c}
	\frac{\partial H}{\partial u \partial \lam}(u^{(0)},1)\\
	\tau ^{(0)^t}\\
	\end{array}\right]$\\	
	$=\det\left[\begin{array}{cc}
	Q^{(0)}_1 & Q^{(0)}_2\\
	(Q^{(0)}_2)^t(Q^{(0)}_1)^{(-T)} & -1\\	
	\end{array}\right]$ \\ $= \det\left[\begin{array}{cc}
	Q^{(0)}_1 & Q^{(0)}_2\\
	0 & -1-(Q^{(0)}_2)^T(Q^{(0)}_1)^{(-T)}(Q^{(0)}_1)^{(-1)}Q_2^{(0)} \\	\end{array}\right] \\$ 
	\vsp
	\NI $=\det(Q^{(0)}_1) \det(-1-(Q^{(0)}_2)^T(Q^{(0)}_1)^{(-T)}(Q^{(0)}_1)^{(-1)}Q_2^{(0)})$ \\ 
	\vsp
	\NI $=-\det(Q^{(0)}_1) \det(1+(Q^{(0)}_2)^T(Q^{(0)}_1)^{(-T)}(Q^{(0)}_1)^{(-1)}Q_2^{(0)})$ \\ 
	\vsp
	\NI $=-\prod_{i=1}^{n}x^{(0)}_i y^{(0)}_i \det(1+(Q^{(0)}_2)^t(Q^{(0)}_1)^{(-T)}(Q^{(0)}_1)^{(-1)}Q_2^{(0)}) <0. $ \\
	\vsp
	\NI So the positive predictor direction $\tau ^{(0)}$ at the initial point $u^{(0)}$ satisfies\\ $\det\left[\begin{array}{c}
	\frac{\partial H}{\partial u \partial \lam}(u^{(0)},1)\\
	\tau ^{(0)^T}\\
	\end{array}\right]<0.$
\end{proof}
\vsp
\begin{remk}
	We conclude from the theorem \ref{direction} that the positive tangent direction $\tau$ of the homotopy path $\Gamma_u^{(0)}$ at any point $(u,t)$ be negative and it depends on det$(Q_1),$ where $Q_1=\left[\begin{array}{ccc} 
	(1-t)(A+A^T)+t I & -(1-t)I & -(1-t)A^T \\
	Y_1+(1-t)(Y+XA) & X & 0 \\
	Y_2A & 0 & Y\\ 
	\end{array}\right].$
\end{remk}
  Based on the earlier work  the homotopy continuation method to solve the initial value problem \ref{sssss}  was formulated with the iterative process as 
	\begin{center}
	    	$I_{j+1}=I_j-\frac{\partial}{\partial u}H(I_j,t+h)^+H(I_j,t+h),$ for $j=0,1,2,\cdots,m-1.$
	    		\end{center}
	    		For details see \cite{abbott1977note}. However the proposed modified homotopy continuation method solves homotopy function by solving the initial value problem \ref{sssss} with the following iterative process 
	    		\begin{center}
	    		   $K_j=\frac{\partial}{\partial u}H(I_j,t+h)^+H(I_j,t+h)$\\ $L_j=I_j-K_j$\\  $KK_j=(\frac{\partial}{\partial u}H(L_j,t+h)+\frac{\partial}{\partial u}H(I_j,t+h))^+H(I_j,t+h)$\\ $LL_j=I_j-2*KK_j$\\
	    	     	    	  $I_{j+1}=LL_j-\frac{\partial}{\partial u}H(L_j,t+h)^+H(L_j,t+h),$ for $j=0,1,2,\cdots,m-1.$ 
	    	  \end{center} 
	    		By this iterative process the proposed homotopy function achieves the order of convergence as $5^m -1.$
 \begin{theorem}
Suppose that the homotopy function has derivative, which is lipschitz continuous in a convex neighbourhood $\cal N$ of $c,$ where $c$ is the solution of the homotopy function $H(u,t)=0,$ whose Jacobian matrix is continuous and nonsingular and bounded on $\cal N.$ 
Then the modified homotopy continuation method has order $5^{m}-1.$
\end{theorem}
\begin{proof}
By the Implicit Function Theorem ensures the
existence of a unique continuous solution $z(h) \in \cal N$ of $\Dot{z}(h)=-\tilde{J}^{-1}\tilde{f},$  $z(0)=u$ and $h\in(-\delta,\delta),$ for some $\delta >0.$ Define $\beta _j=\|z(h)-I_j(u,h)\|.$ From lemma \ref{coorder} $\beta_j=O(h^{5^j}).$ Then $\beta _{j+1}=\|z(h)-I_{j+1}\| \leq K{\beta _j}^5.$ Hence $\beta _{j+1}=O(h^{5^{j+1}}).$ By induction method the modified homotopy continuation method has convergency of order $5^{m}-1$
\end{proof}
\subsection{Solving Discounted Zero-Sum Stochastic Game with ARAT Structure}
\begin{examp}

Consider a two player zero-sum discounted ARAT game with $s = 2$ states.
In each state each  player has $2$ actions. The transition probabilities are given
by\\
$p^1 _1(1,1) = \frac{1}{2}, \ p^1 _1(1,2) = 0,\\$
$p^1 _2(1,1) = \frac{1}{2}, \  p^1 _2(1,2) = 0,\\$
$p^1 _1(2,1) = 0, \ p^1 _1(2,2) = \frac{1}{2},\\$
$p^1 _2(2,1) = 0, \ p^1 _2(2,2) = \frac{1}{2},\\$
$p^2 _1(1,1) =\frac{1}{2}, \ p^2 _1(1,2) = 0,\\$
$p^2 _2(1,1) = 0, \  p^2 _2(1,2) = \frac{1}{2},\\$
$p^2 _1(2,1) = 0, \ p^2 _1(2,2) = \frac{1}{2},\\$
$p^2 _2(2,1) = \frac{1}{2},$ \ $ p^2 _2(2,2) = 0.\\$
Note that $p_{ij}(s,s') = p^1 _i(s,s') + p^2 _j(s,s').$  \\
$P_1= P_1(s)=((p^1_i(s,s'), s, s' \in S, i \in A_s))$ and \\ $P_2=P_2(s)=((p^2_j(s,s'), s, s' \in S, j \in B_s)).$ \\Let the discount factor $\beta =\frac{1}{2}.$  \\ The reward structure:\\
$r^1 _1(1) = 4, \ \ r^1 _1(2) = 5,$\\
$r^1 _2(1) = 3$, \ \  $r^1 _2(2) = 4,$\\
$r^2 _1(1) = 3, \ \ r^2 _1(2) = 6,$\\
$r^2 _2(1) = 6$, \ \ $r^2 _2(2) = 2.$\\ 
Note that $r(s,i, j) = r^1 _i(s) + r^2 _j(s).$\\  
	 Now we solve discounted ARAT game using the proposed homotopy function. The initial point is $u^{(0)}=(4,5,3,4,8,8,6,2,1,1,1,1,1,1,1,1,1,1,1,1,1,1,1,1,0)^T$. As $t \to 0, $ $u=(0,7,0,6,9,0,0,7.33333,1,0,1,0,0,2.33333,3.33333,0,0,7,0,6,9,0,0,7.33333,0)$. Hence the solution of discounted ARAT is $x=\left[\begin{array}{c} 
	1\\
	0\\
	1\\
	0\\
	0\\
	2.33333\\
	3.33333\\
	0\\
	\end{array}\right].$
\end{examp}
		
\begin{examp}
Consider another two player zero-sum discounted ARAT game with $s = 2$ states.
In each state each  player has $2$ actions. The transition probabilities are given
by\\
$p^1 _1(1,1) = \frac{1}{4}, \ p^1 _1(1,2) = 0,\\$
$p^1 _2(1,1) = \frac{1}{4}, \  p^1 _2(1,2) = 0,\\$
$p^1 _1(2,1) = 0, \ p^1 _1(2,2) = \frac{1}{2},\\$
$p^1 _2(2,1) = 0, \ p^1 _2(2,2) = \frac{1}{2},\\$
$p^2 _1(1,1) =\frac{3}{4}, \ p^2 _1(1,2) = 0,\\$
$p^2 _2(1,1) = 0, \  p^2 _2(1,2) = \frac{3}{4},\\$
$p^2 _1(2,1) = 0, \ p^2 _1(2,2) = \frac{1}{2},\\$
$p^2 _2(2,1) = \frac{1}{2}$, \ $ p^2 _2(2,2) = 0.\\$
Note that $p_{ij}(s,s') = p^1 _i(s,s') + p^2 _j(s,s').$ \\
 $P_1= P_1(s)=((p^1_i(s,s'), s, s' \in S, i \in A_s))$ and \\ $P_2=P_2(s)=((p^2_j(s,s'), s, s' \in S, j \in B_s)).$ \\ Let the discount factor $\beta =\frac{1}{2}.$ \\
 The reward structure:\\
 $r^1 _1(1) = 4, \ \ r^1 _1(2)= 5,$\\
$ r^1 _2(1)= 3$, \ \  $r^1 _2(2)  = 4,$\\
 $r^2 _1(1) = 3, \ \ r^2 _1(2) = 6,$\\
 $r^2 _2(1) = 6$, \ \  $r^2 _2(2) = 2.$\\ Note that $r(s,i, j) = r^1 _i(s) + r^2 _j(s).$\\ Now we solve discounted ARAT game using the proposed homotopy function. The initial point is $u^{(0)}=(1,1,1,1,20,20,10,10,1,1,1,1,1,1,1,1,1,1,1,1,1,1,1,1,0)^T$. As $t\to 0,$  $u=(0,9,0,6,7,0,0,7.33333,1,0,1,0,0,2,3.33333,0,0,9,0,6,7,0,0,7.33333,0)$. Hence the solution of discounted ARAT is $x=\left[\begin{array}{c} 
	1\\
	0\\
	1\\
	0\\
	2\\
	3.33333\\
	0\\
0\\
	\end{array}\right].$
\end{examp}
\section{Conclusion}
In this paper, we introduce a homotopy continuation method to find  the solution of discounted ARAT stochastic game. Mathematically, we obtain the positive tangent direction of the homotopy path. We prove that the smooth curve of the proposed homotopy function is  bounded and convergent. We establish that the proposed homotopy functions has $5^m -1$ order of convergence. Two numerical examples are illustrated to demonstrate the effictiveness  of the proposed homotopy function. 
\section{Acknowledgment}
The author A. Dutta is thankful to the Department of Science and Technology, Govt. of India, INSPIRE Fellowship Scheme for financial support. 
\vsp
\bibliographystyle{plain}
\bibliography{output}
\end{document}